\documentclass[12pt,oneside,reqno]{amsart}
\usepackage[all]{xy}
\usepackage{amsfonts,amsmath,oldgerm,amssymb,amscd, comment,multirow}
\UseComputerModernTips
\numberwithin{equation}{section}
\usepackage[breaklinks]{hyperref}
\allowdisplaybreaks

\usepackage{enumerate}

\usepackage{caption} 
\newtheorem{theorem}{Theorem}[section]

\newtheorem{lemma}[theorem]{{\bf Lemma}}
\newtheorem{coro}[theorem]{{\bf Corollary}}
\newtheorem{definition}[theorem]{Definition}

\newcommand{\al}{\alpha}

\newcommand{\Z}{\mbox{$\mathbb Z$}}

\begin{document}

	\title[ Arithmetic density and congruences of $t$-core partitions ]{ Arithmetic density and congruences of $t$-core partitions} 
	
	\author[N.K. Meher]{N.K. Meher}
	\address{Nabin Kumar Meher, Department of Mathematics, Indian Institute of Information Raichur, Govt. Engineering College Campus, Yermarus, Raichur, Karnataka, India 584135.}
	\email{mehernabin@gmail.com, nabinmeher@iiitr.ac.in}
	
	\author[ Ankita Jindal]{ Ankita Jindal}
	\address{Ankita Jindal, Indian Statistical Institute Bangalore, 8th Mile, Mysore Road, Bangalore, Karnataka, India 560059}
	\email{ankitajindal1203@gmail.com }

	\maketitle
	\pagenumbering{arabic}
	\pagestyle{headings}
	\begin{abstract}
		A partition of $n$ is called a $t$-core partition if none of its hook number is divisible by $t.$ In 2019, Hirschhorn and Sellers \cite{Hirs2019} obtained a parity result for $3$-core partition function $a_3(n)$. Recently, both authors \cite{MeherJindal2022} proved density results for $a_3(n)$, wherein we proved that $a_3(n)$ is almost always divisible by arbitrary power of $2$ and $3.$ In this article, we prove that for a non-negative integer $\alpha,$ $a_{3^{\alpha} m}(n)$ is almost always divisible by arbitrary power of $2$ and $3.$ Further, we prove that $a_{t}(n)$ is almost always divisible by arbitrary power of $p_i^j,$ where $j$  is a fixed positive integer and $t= p_1^{a_1}p_2^{a_2}\ldots p_m^{a_m}$ with primes $p_i \geq 5.$ Furthermore, by employing Radu and Seller's approach, we obtain an algorithm and we give alternate proofs of several congruences modulo $3$ and $5$ for $a_{p}(n)$, where $p$ is prime number. Our results also generalizes the results in  \cite{radu2011a}.
	\end{abstract}
	\maketitle
	
	\section{Introduction}
	A partition $\al=(\al_1,\al_2, \cdots, \al_{s})$ of $n$ is a non-increasing sequence of positive integers whose sum is $n$ and the positive integers $\al_i$ are called parts of the partitions. A partition $\al$ of $n$ can be represented by the Young diagram $[\al]$ (also known as the Ferrers graph) which consists of the $s$ number of rows such that the $i^{th}$ row has $\al_i$ number of dots $\bullet$ and all the rows start from the same column. An illustration of the Young diagram for $\al=(\al_1,\al_2, \cdots, \al_{s})$ is as follows.
	\begin{center}
		$[\al]$:=\begin{tabular}{lcll}
			$\bullet$ & $\bullet$& $\cdots$$\cdots$$\cdots$\hspace{0.3cm}$\bullet$ & $\al_1$ dots \\
			$\bullet$ & $\bullet$& $\cdots$$\cdots$\hspace{0.3cm}$\bullet$ & $\al_2$ dots \\
			& $\vdots$& & \vdots \\
			$\bullet$ &$\bullet$ & $\cdots$\hspace{0.3cm}$\bullet$ & $\al_s$ dots
		\end{tabular}
	\end{center}
	For $1\leq i \leq s$ and $1\leq j \leq \al_i$, the dot of $[\al]$ which lies in the $i^{th}$ row and $j^{th}$ column is denoted by $(i,j)^{th}$-dot of $[\al]$. Let $\al_j^{'}$ denotes the number of dots in the $j^{th}$ column. The hook number $H_{i,j}$ of $(i,j)^{th}$-dot is defined by $\al_i+\al_j^{'}-i-j+1$. In other words, $H_{i,j}=1+h_0$ where $h_0$ is the sum of the number of dots lying to the right of the $(i,j)^{th}$-dot in the $i^{th}$ row, and the number of dots lying below the $(i,j)^{th}$-dot in the $j^{th}$ column. Given a partition $\al$ of $n$, we say that it is a $t$-core partition if none of its hook numbers is divisible by $t$. 
	
	\noindent\textbf{Example.} The Young diagram of the partition $\al=(5,3,2)$ of $10$ is
	\begin{center}
		\begin{tabular}{ccccc}
			$\bullet^{7}$ & $\bullet^{6}$& $\bullet^{4}$ & $\bullet^{2}$ & $\bullet^{1}$ \\
			$\bullet^{4}$ & $\bullet^{3}$& $\bullet^{1}$ & & \\
			$\bullet^{2}$ & $\bullet^{1}$ & & &
		\end{tabular}
	\end{center}
	where the superscripts on each dot represents its hook number. It can be easily observed that $\alpha$ is a $t$-core partition of $10$ for $t=5$ and $t \geqslant 8$.
	
	For $n\geq 0$, let $a_t(n)$ denote the number of partitions of $n$ that are $t$-core partitions.  The generating function for $a_t(n)$ is given by
	\begin{equation}\label{1e1}
		\sum_{n=0}^{\infty} a_t(n) q^n = \prod\limits_{n=1}^{\infty} \frac{(1-q^{nt})^t}{(1-q^n)}= \frac{(q^t;q^t)^t_{\infty}}{(q;q)_{\infty}}
	\end{equation}
	where $(a;q)_{n}=(1-a)(1-aq)(1-aq^2)\cdots(1-aq^n)$ and $(a;q)_{\infty}=\lim \limits_{n \to \infty}(a;q)_{n}$. 
	
	In \cite[Corollary 1]{Garvan1990}, Garvan, Kim and Stanton proved the congruence
	\begin{align}\label{eq5711}
		a_p(p^jn-\delta_{p})\equiv 0\pmod {p^j}
	\end{align}
	where $p \in \{5,7,11\}$, $n$ and $j$ are positive integers and $\delta_{p}= \frac{p^2-1}{24}$. In \cite[Proposition 3]{Graville1996}, Granville and Ono proved similar congruences, namely
	\begin{align*}
		a_{5^j}(5^jn-\delta_{5,j})&\equiv 0\pmod {5^j},\\
		a_{7^j}(7^jn-\delta_{7,j})&\equiv 0\pmod {7^{\lfloor \frac j2 \rfloor +1}},\\
		a_{11^j}(11^jn-\delta_{11,j})&\equiv 0\pmod {11^j}
	\end{align*}
	where $n$ and $j$ are positive integers and $\delta_{p,j}\equiv \frac{1}{24}\pmod{p^j}$ for $p \in \{5,7,11\}$.
	
	Recently, Hirschhorn and Sellers \cite{Hirs2019} proved that, for all $n \geq 0$,
	\begin{align*}
		a_3(n)=\begin{cases}
			1\pmod 2	& \textrm{if }n= 3m^2+2m \textrm{ for some integer } m,\\
			0 \pmod 2	& \textrm{otherwise.}
		\end{cases}
	\end{align*}
	In \cite{MeherJindal2022}, the authors proved that the set $ \{ n \in \mathbb{N}: a_3(n) \equiv 0 \pmod{p^j}\}$ has arithmetic density $1$ for $p \in \{2,3\}$. In this article, we study the arithmetic densities of the partition function $a_t(n)$ modulo arbitrary powers of $2$ and $3$ when $t= 3^{\al}m$, where $\al\geq 0$, $m\geq 1$ are integers with $\gcd(m,6)=1$, and modulo arbitrary prime powers $p_i^j$ when  $t=p_1^{a_1}p_2^{a_2}\ldots p_m^{a_m},$ where $p_i \geq 5$
	are prime numbers. Precisely, we prove the following results.
	\begin{theorem}\label{mainthm1} Let $j\geq 1$, $\al\geq 0$ and $m\geq 1$ be integers with $\gcd(m,6)=1$. Then the set 
		$$ \{ n \in \mathbb{N}: a_{3^{\alpha}m}(n) \equiv 0 \pmod{2^j}\}$$ has arithmetic density $1$.
	\end{theorem}

	\begin{theorem}\label{mainthm2} Let $j\geq 1$, $\al\geq 0$ and $m\geq 1$ be integers with $\gcd(m,6)=1$. Then the set 
		$$ \{ n \in \mathbb{N}: a_{3^{\alpha}m}(n) \equiv 0 \pmod{3^j}\}$$ 
	has arithmetic density $1.$
	\end{theorem}
	
	\begin{theorem}\label{mainthm11} For a positive integer $m,$ let $a_1, a_2, \ldots, a_m$ be non negative integers. Let $t=p_1^{a_1}p_2^{a_2}\ldots p_m^{a_m},$ where $p_i \geq 5$'s
	are prime numbers. Then for every positive integer $j,$  the set 
	$$\left\{0 < n < x: a_{t}(n)\equiv 0 \pmod{p_i^j} \right\}$$ has arithmetic density $1.$
	\end{theorem}
As a consequence of the above theorem, we obtain the following result.
\begin{coro}\label{coro11}
	Let $j$ be a positive integer and $ p \geq 5$ be a prime number. Then $a_p(n)$ is almost always is divisible by $p^j,$ namely
	$$ \lim\limits_{X \to \infty} \frac{\# \left\{0 < n \leq X: a_p(n) \equiv 0 \pmod{p^j}\right\}}{X} =1.$$ 	
\end{coro}

	In the remaining part of this section, we give an algorithm for $p$-core partitions for primes $p\geq 3$ and we obtain several congruences for $p$-core partitions modulo 3 and modulo 5 for $5\leq p \leq 23$ using this algorithm. We use the technique given by Radu and Sellers \cite{Radu2011}. Before proceeding further, we define some notation. For an integer $m$, a prime $p\geqslant 3$ and $t\in \{0,1,\dots,m-1\}$, we set
	\begin{align*}
		\kappa:=&\kappa(m)=\gcd(m^2-1,24),\\
		\hat {p}:=&\frac{p^2-1}{24},\\
		A_t:=&A_t(m,p)=\frac{24m}{\gcd(-\kappa(24 t+p^2-1),24m)} =\begin{cases}
			\frac{m}{\gcd(-\kappa (t+\hat{p}),m)} & \textrm{ if } p>3,\\
			\frac{3m}{\gcd(-\kappa (3t+1),3m)}& \textrm{ if } p=3,
		\end{cases}\\
		\epsilon_2:=&\epsilon_2(m,p)=\begin{cases}
			\frac{1-(-1)^{\frac{p-1}{2}}}{2} & \textrm{ if } 2| m,\\
			0 & \textrm{ if } 2\nmid m,
		\end{cases}\\
		\epsilon_p:=&\epsilon_p(m)=\begin{cases}
			1 & \textrm{ if } p\nmid m,\\
			0 & \textrm{ if } p|m.
		\end{cases}
	\end{align*}
	We note that $\hat {p}\in \Z$ for $p\geq 5$. Also, it is immediate that $A_t \in \Z$ for each $t\in \{0,1,\dots,m-1\}$.
	
	\begin{theorem}\label{mainthmRaduSellers}
		Let $p\geq 3$ be a prime and let $u$ be an integer. For a positive integer $g$, let $e_1, e_2, \cdots, e_g$ be non-negative integers. Let $m=p_1^{e_1} p_2^{e_2} \cdots p_g^{e_g}$ where $p_i$'s are prime numbers. Let $t \in \{0,1,\ldots, m-1\}$ such that $A_t$ divides $2^{\epsilon_2} p^{\epsilon_p}p_1p_2\cdots p_{g}$. Define
		\begin{align*}
			P(t):=\left\{t^{'}: \exists [s]_{24m} \textrm{ such that } t^{'}\equiv ts+\frac{(s-1)(p^2-1)}{24}  \pmod {m}\right\}
		\end{align*} 
		where $[s]_{24m}$ is the residue class of $s$ in $\mathbb{Z}_{24m}$. If the congruence $a_p(mn+t^{'}) \equiv 0 \pmod u$ holds for all $t^{'} \in P(t)$ and $ 0\leq n \leq \left\lfloor \frac{2^{\epsilon_2}(p+1)^{\epsilon_p}(p-1)(p_1+1)(p_2+1)\cdots (p_{g}+1)}{24}-\frac{(p^2-1)}{24m} \right\rfloor$, then $a_p(mn+t^{'}) \equiv 0 \pmod u$ holds for all $t^{'} \in P(t)$ and $n \geq 0$.
	\end{theorem}
	
	Let $m=p p_1p_2 \cdots p_g$ be a square-free integer  where $p$ and $p_i$'s are prime numbers. Since $p|m$, we have $\epsilon_p=0$. If $p=3$, then $3\nmid\kappa$ gives $9|A_t$ which implies that $A_t\nmid  2^{\epsilon_2} pp_1p_2\cdots p_{g}$ for each $t \in \{0,1,\ldots, m-1\}$. If $p>3$, then $A_t|  2^{\epsilon_2} pp_1p_2\cdots p_{g}$ for each $t \in \{0,1,\ldots, m-1\}$. Thus we have the following corollary to Theorem \ref{mainthmRaduSellers}.
	
	\begin{coro}\label{squarefreem}
		Let $p\geq 5$ be a prime and let $u$ be an integer. For a positive integer $g$, let $m=p p_1p_2 \cdots p_g$ where $p$ and $p_i$'s are distinct prime numbers. Let $t \in \{0,1,\ldots, m-1\}$. Define
		\begin{align*}
			P(t):=\left\{t^{'}: \exists [s]_{24m} \textrm{ such that } t^{'}\equiv ts+(s-1)\hat{p}  \pmod {m}\right\}
		\end{align*}
		where $[s]_{24m}$ is the residue class of $s$ in $\mathbb{Z}_{24m}$. If the congruence $a_p(mn+t^{'}) \equiv 0 \pmod u$ holds for all $t^{'} \in P(t)$ and $ 0\leq n \leq \left\lfloor \frac{2^{\epsilon_2}(p^2-1)(p_1+1)(p_2+1)\cdots (p_{g}+1)}{24}-\frac{(p^2-1)}{24m} \right\rfloor\leq 2^{\epsilon_2}\hat{p}(p_1+1)(p_2+1)\cdots (p_{g}+1)-1$, then $a_p(mn+t^{'}) \equiv 0 \pmod u$ holds for all $t^{'} \in P(t)$ and $n \geq 0$. 
	\end{coro}

	In order to further simplify, we next consider the particular case of a square-free integer $m$ of the type $m=pq$  in Corollary \ref{squarefreem}. When we consider $q=2$, we obtain Corollary \ref{coro2p} and when we consider $q\geq 3$, we deduce Corollary \ref{coropq}. 
	
	\begin{coro}\label{coro2p}
		Let $p\geq 5$ be a prime and let $u$ be an integer. Let $t \in \{0,1,\ldots, 2p-1\}$. Define
		\begin{align*}
			P(t):=\left\{t^{'}: \exists [s]_{48p} \textrm{ such that } t^{'}\equiv ts+(s-1) \hat{p}  \pmod {48p}\right\}
		\end{align*} 
		where $[s]_{48p}$ is the residue class of $s$ in $\mathbb{Z}_{48p}$. If the congruence $a_p(2pn+t^{'}) \equiv 0 \pmod u$ holds for all $t^{'} \in P(t)$ and $ 0\leq n \leq 2^{\frac{1-(-1)^{\frac{p-1}{2}}}{2}} 3 \hat{p}-1$, then $a_p(2pn+t^{'}) \equiv 0 \pmod u$ holds for all $t^{'} \in P(t)$ and $n \geq 0$.
	\end{coro}
	
	In \cite[Theorem 1.4]{radu2011a}, Radu and Sellers obatined several congruences of the form $a_p(2pn+t)\equiv 0 \pmod 3$ for $5\leq p \leq 23$ which can also be obtained from the above corollary by performing a finite check for certain initial values of $n$. They used a similar approach as in this paper to obtain their results. In fact, they simplified the set $P(t)$ in case of $m=2p$ as
	\begin{align*}
		P(t)= \left \{t^{'}: \left(\frac{24t-1}{24}\right)=\left(\frac{24t^{'}-1}{24}\right), t \equiv t^{'}\pmod 2, 0 \leq t^{'} \leq 2p-1 \right\}.
	\end{align*}
	Also, they connected there results with broken $k$-diamond partitions. Our result Theorem \ref{mainthmRaduSellers} is generalizes their work. 
	
	\begin{coro}\label{coropq}
		Let $p,q$ be two distinct primes with $p\geq5$ and $q\geq 3$ and let $u$ be an integer. Let $t \in \{0,1,\ldots, pq-1\}$. Define
		\begin{align*}
			P(t):=\left\{t^{'}: \exists [s]_{24pq} \textrm{ such that } t^{'}\equiv ts+(s-1) \hat{p}  \pmod {pq}\right\}
		\end{align*} 
		where $[s]_{24pq}$ is the residue class of $s$ in $\mathbb{Z}_{24pq}$. If the congruence $a_p(pqn+t^{'}) \equiv 0 \pmod u$ holds for all $t^{'} \in P(t)$ and $ 0\leq n \leq \frac{(p^2-1)(q+1)}{24} =\hat{p}(q+1)$, then $a_p(pqn+t^{'}) \equiv 0 \pmod u$ holds for all $t^{'} \in P(t)$ and $n \geq 0$.
	\end{coro}
	
	We obtain Theorems \ref{congruences} and \ref{congruences5} as applications of Corollary \ref{coropq}.
	
	\begin{theorem}\label{congruences}
		For all $n\geq 0$, we have
		\begin{align*}
			&a_{5}(15n+6,10,12,13) \equiv 0 \pmod 3,\\
			&a_{7}(21n+3, 8, 11, 15, 17, 18) \equiv 0 \pmod 3,\\
			&a_{11}(33n+ 3,11,12, 20, 24, 26, 27, 29 ,30,32) \equiv 0 \pmod 3,\\
			&a_{13}(39n+3,7,9,10,15,16,18,22,28,31,33,36) \equiv 0 \pmod 3,\\
			&a_{17}(51n+10,14,16,19,20,23,25,26,28,34,35, 38,41,46,47,49) \equiv 0 \pmod 3,\\
			&a_{19}(57n+7,14,16,17,19,22,25,26,31,35,37,38,41,44,50,52,55,56) \equiv 0 \pmod 3,\\
			&a_{23}(69n+3,9,16,22,27,30,31,33,34,36,42,43,46,48,51,52,58,60,61,63,64,64,67) \\
			&\hspace{
				13cm}\equiv 0 \pmod 3.
		\end{align*}
	\end{theorem}
	
	\begin{theorem}\label{congruences5}
		For all $n\geq 0$, we have
		\begin{align*}
			&a_{7}(35n+4, 17, 22, 24, 29, 32) \equiv 0 \pmod 5,\\
			&a_{11}(55n+4,7, 9, 11, 18, 21, 22, 26, 29, 31, 32, 33, 37, 42, 43, 44, 48, 51, 53, 54) \\
			&\hspace{
				11cm}\equiv 0 \pmod 5,\\
			&a_{17}(85n+4, 7, 9, 14, 24, 37, 47, 52, 54, 57, 64, 69, 72, 74, 77, 82) \equiv 0 \pmod 5.
		\end{align*}
	\end{theorem}

	We also obtain a number of congruences for $a_p(n)$ modulo $3$ for $p\in\{5,7,11,13,19\}$ in the following result using Theorem \ref{mainthmRaduSellers}.
	\begin{theorem}\label{congruencesrandom} 
		For all $n\geq 0$, we have
		\begin{align}
			&a_{5}(8n+3) \equiv 0 \pmod 3, \label{18e1}\\
			&a_{5}(49n+6, 13, 20,27,34,41) \equiv 0 \pmod 3,\label{18e2}\\
			&a_{7}(25n+3, 8, 13, 18) \equiv 0 \pmod 3,\label{18e3}\\		
			&a_{11}(44n+ 9,21,29,33,37) \equiv 0 \pmod 3,\label{18e4}\\
			&a_{13}(12n+3) \equiv 0 \pmod 3,\label{18e5}\\
			&a_{13}(64n+25) \equiv 0 \pmod 3,\label{18e6}\\
			&a_{19}(76n+3, 7, 19, 31, 35, 55, 63, 71, 75) \equiv 0 \pmod 3.\label{18e7}
		\end{align}
	\end{theorem}

	Alternate proofs of the congruences obtained in Theorems \ref{congruences}, \ref{congruences5} and \ref{congruencesrandom} can be found in Garvan \cite{Garvan1993} and Chen \cite{Chen2013}. Further Theorem \ref{mainthmRaduSellers} can be used to obtain an alternate proof of the congruences
\begin{align*}
	a_5(5n+4) &\equiv 0 \pmod {5},\\
	a_7(7n+5) &\equiv 0 \pmod {7},\\
	a_{11}(11n+6) &\equiv 0 \pmod {11}
\end{align*}
for $n\geq 0$, which coincide with the congruences \eqref{eq5711} with $j=1$.

	Our paper is of two folds. Firstly, we prove density results for the parity of the partition function $a_{3^{\alpha}m}(n)$ in Theorems \ref{mainthm1} and \ref{mainthm2} using modular form techniques and Serre's result. Secondly, we provide an algorithm using Radu and Seller's method and we also obtain some algebraic results in Theorems \ref{congruences} and \ref{congruencesrandom} which supports our analytical results.
	
	\section{Preliminaries}
	We recall some basic facts and definition on modular forms. For more details, one can see \cite{Koblitz}, \cite{Ono2004}. We start with some matrix groups. We define
	\begin{align*}
		\Gamma:=\mathrm{SL_2}(\mathbb{Z})= &\left\{ \begin{bmatrix}
			a && b \\c && d
		\end{bmatrix}: a, b, c, d \in \mathbb{Z}, ad-bc=1 \right\},\\
		\Gamma_{\infty}:= &\left\{\begin{bmatrix}
			1 &n\\ 0&1	\end{bmatrix}: n \in \mathbb{Z}\right\}.
	\end{align*}
	For a positive integer $N$, we define
	\begin{align*}
		\Gamma_{0}(N):=& \left\{ \begin{bmatrix}
			a && b \\c && d
		\end{bmatrix} \in \mathrm{SL_2}(\mathbb{Z}) : c\equiv0 \pmod N \right\},\\
		\Gamma_{1}(N):=& \left\{ \begin{bmatrix}
			a && b \\c && d
		\end{bmatrix} \in \Gamma_{0}(N) : a\equiv d  \equiv 1 \pmod N \right\}
	\end{align*}
	and 
	\begin{align*}
		\Gamma(N):= \left\{ \begin{bmatrix}
			a && b \\c && d
		\end{bmatrix} \in \mathrm{SL_2}(\mathbb{Z}) : a\equiv d  \equiv 1 \pmod N,  b \equiv c  \equiv 0 \pmod N \right\}.
	\end{align*}
	A subgroup of $\Gamma=\mathrm{SL_2}(\mathbb{Z})$ is called a congruence subgroup if it contains $ \Gamma(N)$ for some $N$ and the smallest $N$ with this property is called its level. Note that $ \Gamma_{0}(N)$ and $ \Gamma_{1}(N)$ are congruence subgroup of level $N,$ whereas $ \mathrm{SL_2}(\mathbb{Z}) $ and $\Gamma_{\infty}$ are congruence subgroups of level $1.$ The index of $\Gamma_0(N)$ in $\Gamma$ is 
	\begin{align*}
		[\Gamma:\Gamma_0(N)]=N\prod\limits_{p|N}\left(1+\frac 1p\right)
	\end{align*}
	where $p$ runs over the prime divisors of $N$.
	
	Let $\mathbb{H}$ denote the upper half of the complex plane $\mathbb{C}$. The group 
	\begin{align*}
		\mathrm{GL_2^{+}}(\mathbb{R}):= \left\{ \begin{bmatrix}
			a && b \\c && d
		\end{bmatrix}: a, b, c, d \in \mathbb{R}, ad-bc>0 \right\},
	\end{align*}
	acts on $\mathbb{H}$ by $ \begin{bmatrix}
		a && b \\c && d
	\end{bmatrix} z = \frac{az+b}{cz+d}.$ We identify $\infty$ with $\frac{1}{0}$ and define $ \begin{bmatrix}
		a && b \\c && d
	\end{bmatrix} \frac{r}{s} = \frac{ar+bs}{cr+ds},$ where $\frac{r}{s} \in \mathbb{Q} \cup \{ \infty\}$. This gives an action of $\mathrm{GL_2^{+}}(\mathbb{R})$ on the extended half plane $\mathbb{H}^{*}=\mathbb{H} \cup \mathbb{Q} \cup \{\infty\}$. Suppose that $\Gamma$ is a congruence subgroup of $\mathrm{SL_2}(\mathbb{Z})$. A cusp of $\Gamma$ is an equivalence class in $\mathbb{P}^{1}=\mathbb{Q} \cup \{\infty\}$ under the action of $\Gamma$.
	
	The group $\mathrm{GL_2^{+}}(\mathbb{R})$ also acts on functions $g:\mathbb{H} \rightarrow \mathbb{C}$. In particular, suppose that $\gamma=\begin{bmatrix}
		a && b \\c && d
	\end{bmatrix}\in \mathrm{GL_2^{+}}(\mathbb{R})$. If $g(z)$ is a meromorphic function on $\mathbb{H}$ and $k$ is an integer, then define the slash operator $|_{k}$ by
	\begin{align*}
		(g|_{k} \gamma)(z):= (\det \gamma)^{k/2} (cz+d)^{-k} g(\gamma z).
	\end{align*}
	
	\begin{definition}
		Let $\Gamma$ be a congruence subgroup of level $N$. A holomorphic function $g:\mathbb{H} \rightarrow \mathbb{C}$ is called a modular form with integer weight $k$ on $\Gamma$ if the following hold:
		\begin{enumerate}[$(1)$]
			\item We have
			\begin{align*}
				g \left( \frac{az+b}{cz+d}\right)=(cz+d)^{k} g(z)
			\end{align*}
			for all $z \in \mathbb{H}$ and $\begin{bmatrix}
				a && b \\c && d
			\end{bmatrix}\in \Gamma$. 
			\item If $\gamma\in SL_2 (\mathbb{Z})$, then $(g|_{k} \gamma)(z)$ has a Fourier expnasion of the form
			\begin{align*}
				(g|_{k} \gamma)(z):= \sum \limits_{n\geq 0}a_{\gamma}(n) q_N^{n}
			\end{align*}
			where $q_N:=e^{2\pi i z /N}$.
		\end{enumerate}
	\end{definition}
	For a positive integer $k$, the complex vector space of modular forms of weight $k$ with respect to a congruence subgroup $\Gamma$ is denoted by $M_{k}(\Gamma)$.
	
	\begin{definition} \cite[Definition 1.15]{Ono2004}
		If $\chi$ is a Dirichlet character modulo $N$, then we say that a modular form $g \in M_{k}(\Gamma_1(N))$ has Nebentypus character $\chi$ if 
		\begin{align*}
			g \left( \frac{az+b}{cz+d}\right)=\chi(d) (cz+d)^{k} g(z)
		\end{align*}
		for all $z \in \mathbb{H}$ and $\begin{bmatrix}
			a && b \\c && d
		\end{bmatrix}\in \Gamma_{0}(N)$. The space of such modular forms is denoted by $M_{k}(\Gamma_0(N), \chi)$.
	\end{definition}
	
	The relevant modular forms for the results obtained in this article arise from eta-quotients. Recall that the Dedekind eta-function $\eta (z)$ is defined by 
	\begin{align*}
		\eta (z):= q^{1/24}(q;q)_{\infty}=q^{1/24} \prod\limits_{n=1}^{\infty} (1-q^n)
	\end{align*}
	where $q:=e^{2\pi i z}$ and $z \in \mathbb{H}$. A function $g(z)$ is called an eta-quotient if it is of the form
	\begin{align*}
		g(z):= \prod\limits_{\delta|N} \eta(\delta z)^{r_{\delta}}
	\end{align*}
	where $N$ and $r_{\delta}$ are integers with $N>0$. 
	
	\begin{theorem} \cite[Theorem 1.64]{Ono2004} \label{thm2.3}
		If $g(z)=\prod\limits_{\delta|N} \eta(\delta z)^{r_{\delta}}$ is an eta-quotient such that $k= \frac 12$ $\sum_{\delta|N} r_{\delta}\in \mathbb{Z}$, 
		\begin{align*}
			\sum\limits_{\delta|N} \delta r_{\delta} \equiv 0\pmod {24}	\quad \textrm{and} \quad \sum\limits_{\delta|N} \frac{N}{\delta}r_{\delta} \equiv 0\pmod {24},
		\end{align*}
		then $g(z)$ satisfies
		\begin{align*}
			g \left( \frac{az+b}{cz+d}\right)=\chi(d) (cz+d)^{k} f(z)
		\end{align*}
		for each $\begin{bmatrix}
			a && b \\c && d
		\end{bmatrix}\in \Gamma_{0}(N)$. Here the character $\chi$ is defined by $\chi(d):= \left(\frac{(-1)^{k}s}{d}\right)$ where $s=\prod_{\delta|N} \delta ^{r_{\delta}}$.
	\end{theorem}
	
	\begin{theorem} \cite[Theorem 1.65]{Ono2004} \label{thm2.4}
		Let $c,d$ and $N$ be positive integers with $d|N$ and $\gcd(c,d)=1$. If $f$ is an eta-quotient satisfying the conditions of Theorem \ref{thm2.3} for $N$, then the order of vanishing of $f(z)$ at the cusp $\frac{c}{d}$ is
		\begin{align*}
			\frac{N}{24}\sum\limits_{\delta|N} \frac{\gcd(d, \delta)^2 r_{\delta}}{\gcd(d, \frac{N}{ d} )d \delta}.
		\end{align*}
	\end{theorem}
	Suppose that $g(z)$ is an eta-quotient satisfying the conditions of Theorem \ref{thm2.3} and that the associated weight $k$ is a positive integer. If $g(z)$ is holomorphic at all of the cusps of $\Gamma_0(N)$, then $g(z) \in M_{k}(\Gamma_0(N), \chi)$. Theorem \ref{thm2.4} gives the necessary criterion for determining orders of an eta-quotient at cusps. In the proofs of our results, we use Theorems \ref{thm2.3} and \ref{thm2.4} to prove that $g(z) \in M_{k}(\Gamma_0(N), \chi)$ for certain eta-quotients $g(z)$ we consider in the sequel.
	
	We shall now mention a result of Serre \cite[P. 43]{Serre1974} which will be used later. 
	
	\begin{theorem}\label{thm2.5}
		Let $g(z) \in M_{k}(\Gamma_{0}(N), \chi)$ has Fourier expansion
		$$ g(z)= \sum_{n=0}^{\infty} b(n) q^n \in \mathbb{Z}[[q]].$$
		Then for a positive integer $r$,  there is a constant $\alpha>0$ such that
		$$\#\{ 0 < n \leq X: b(n) \not \equiv 0 \pmod{r}\} = \mathcal{O}\left( \frac{X}{(\log X)^{\alpha}}\right).$$
		Equivalently 
		\begin{align}\label{2e1}
			\begin{split}
				\lim\limits_{X \to \infty} \frac{\#\{ 0 < n \leq X: b(n) \not \equiv 0 \pmod{r}\}}{X}= 0.
		\end{split}	\end{align}
	\end{theorem}

	We finally recall the definition of Hecke operators and a few relavent results. Let $m$ be a positive integer and $g(z)= \sum \limits_{n= 0}^{\infty}a(n) q^{n}\in M_{k}(\Gamma_0(N), \chi)$. Then the action of Hecke operator $T_m$ on $f(z)$ is defined by
	\begin{align*}
		g(z)|T_{m} := \sum \limits_{n= 0}^{\infty} \left(\sum \limits_{d|\gcd(n,m)} \chi(d) d^{k-1} a\left(\frac{mn}{d^2}\right)\right)q^{n}.
	\end{align*}
	In particular, if $m=p$ is a prime, we have
	\begin{align*}
		g(z)|T_p := \sum \limits_{n= 0}^{\infty}\left( a(pn) + \chi(p) p^{k-1} a\left(\frac{n}{p}\right)\right)q^{n}.
	\end{align*}
	We note that $a(n)=0$ unless $n$ is a non-negative integer.

	%
	
	\section{Proof of Theorem \ref{mainthm1}} \label{pfmainthm1}
	We put $t= 3^{\alpha}m$ in $\eqref{1e1}$ to obtain
	\begin{align}\label{3e1}
		\sum_{n=0}^{\infty} a_{3^{\alpha}m}(n) q^n &= \frac{(q^{3^{\alpha}m};q^{3^{\alpha}m})^{3^{\alpha}m}_{\infty}}{(q;q)_{\infty}}.
	\end{align}
	We define
	$$A_{\alpha, m}(z):= \frac{\eta^{2}(2^3 3^{\alpha +1} m z)}{\eta(2^4 3^{\alpha +1} m z)} = \prod_{n=1}^{\infty} \frac{\left(1-q^{2^3 3^{\alpha +1} mn}\right)^2}{\left(1-q^{2^4 3^{\alpha +1} mn}\right)}.$$
	Note that for any prime $p$ and positive integer $j,$  we have
	\begin{equation}\label{3e2}
		(q;q)_{\infty}^{p^j}\equiv (q^p;q^p)_{\infty}^{p^{j-1}} \pmod{p^j}.
	\end{equation}
	Using the above formula, we get
	\begin{equation}\label{3e3}
		A^{2^j}_{\alpha, m}(z)= \frac{\eta^{2^{j+1}}(2^3 3^{\alpha +1} m z)}{\eta^{2^j}(2^4 3^{\alpha +1} mz)} \equiv 1 \pmod {2^{j+1}}.
	\end{equation}
	We define 
	\begin{equation}\label{eq502a}
		B_{\alpha, m, j}(z)= \frac{\eta^{3^{\alpha}m}(2^3 3^{\alpha +1} m z)}{\eta(24z)}	A^{2^j}_{\alpha, m}(z)=  \frac{\eta^{3^{\alpha}m+ 2^{j+1}}(2^3 3^{\alpha +1} m z)}{ \eta(24z) \eta^{2^j}(2^4 3^{\alpha +1} mz)}. 
	\end{equation}	
	Using \eqref{3e3}, we have
	\begin{align}\label{eq503}
		B_{\alpha, m, j}(z) \equiv \frac{\eta^{3^{\alpha}m}(2^3 3^{\alpha +1} m z)}{\eta(24z)}\equiv q^{3^{2 \alpha}m^2 -1} \frac{\left(q^{2^3 3^{\alpha +1} m}; q^{2^3 3^{\alpha +1} m}\right)_{\infty}^{3^{\alpha}m}}{\left(q^{24};q^{24}\right)_{\infty}} \pmod {2^{j+1}}.
	\end{align}
	From \eqref{3e1} and \eqref{eq503}, we obtain
	\begin{align}\label{3e6}
		B_{\alpha, m, j}(z) \equiv \sum_{n=0}^{\infty} a_{3^{\alpha}m}(n) q^{24n+ 3^{2 \alpha}m^2 -1} \pmod{2^{j+1}}.
	\end{align}
	Next, we prove that $B_{\alpha, m, j}(z)$ is a modular form. Applying Theorem \ref{thm2.3}, we first estimate the level of eta quotient $B_{\alpha, m, j}(z)$ . The level of $ B_{\alpha, m, j}(z) $ is $N=2^43^{\alpha+1}m M,$ where $M$ is the smallest positive integer which satisfies 
	$$ 2^43^{\alpha+1}m M \left[ \frac{3^{\alpha}m+2^{j+1}}{2^33^{\alpha+1}m} - \frac{1}{2^3 3} - \frac{2^{j}}{2^43^{\alpha+1}m} \right]\equiv 0 \pmod{24} \implies 3\cdot2^j M \equiv 0 \pmod{24}.  $$
	Therefore $M=4$ and the level of $B_{\alpha, m, j}(z)$  is $N=2^6 3^{\alpha+1}m$. The cusps of $\Gamma_{0}(2^6 3^{\alpha+1}m)$ are given by fractions $\frac{c}{d}$ where $d|2^6 3^{\alpha+1}m$ and $\gcd(c,d)=1.$ By using Theorem \ref{thm2.4}, we have that $B_{\alpha, m, j} $ is holomorphic at a cusp $\frac{c}{d}$ if and only if 
	\begin{align*}
		& \left(3^{\alpha}m+2^{j+1}\right)\frac{\gcd^2(d, 2^33^{\alpha+1}m)}{2^33^{\alpha+1}m}- \frac{\gcd^2(d, 24)}{24} -2^j \frac{\gcd^2(d, 2^43^{\alpha+1}m)}{2^43^{\alpha+1}m} \geq 0 \\
		& \iff L:= 2 \left(3^{\alpha}m+2^{j+1}\right) G_1 - 2\cdot 3^{\alpha}m G_2- 2^j \geq 0,
	\end{align*}
	where $G_1= \frac{\gcd^2(d, 2^3 3^{\alpha+1}m)}{\gcd^2(d, 2^4 3^{\alpha+1}m)}$ and  $G_2= \frac{\gcd^2(d, 24)}{\gcd^2(d, 2^43^{\alpha+1}m)}.$ 	 Let $d$ be a divisor of $ 2^6 3^{\alpha+1}m$. We can write $d= 2^{r_1} 3^{r_2} t$ where $ 0 \leq r_1 \leq 6$,   $ 0 \leq r_2 \leq \alpha+1$ and $t|m$. We now consider the following two cases depending on $r_1$.
	
	\noindent Case 1: Let $ 0 \leq r_1 \leq 3,$ $ 0 \leq r_2 \leq \alpha+1$. Then $G_1= 1$ and $\frac{1}{3^{2\alpha}t^2} \leq G_2 \leq  1$ which implies
	$ L \geq  2 \left(3^{\alpha}m+2^{j+1}\right)  - 2\cdot 3^{\alpha}m - 2^j = 3\cdot2^j > 0.$
	
	\noindent Case 2: Let $ 4 \leq r_1 \leq 6$, $ 0 \leq r_2 \leq \alpha+1$. Then $G_1=\frac{1}{4}$ and $\frac{1}{4 3^{2\alpha}t^2} \leq G_2 \leq  \frac{1}{4}$. This gives
	$L= 2 \left(3^{\alpha}m+2^{j+1}\right) G_1 - 2\cdot 3^{\alpha}m G_2- 2^j  \geq \frac{3^{\alpha}m}{2}+ 2^j -\frac{3^{\alpha}m}{2} - 2^j=0 .$ 
	
	Therefore, $B_{\alpha, m, j}(z)$ is holomorphic at every cusp $\frac{c}{d}.$ Using Theorem \ref{thm2.3}, we estimate that the weight of $B_{\alpha, m, j}(z)$ is $k= \frac{3^{\alpha}m-1}{2}+ 2^{j-1}$ which is a positive integer. The associated character for $B_{\alpha, m, j}(z)$ is $$\chi= \left( \frac{(-1)^{\frac{3^{\alpha}m-1}{2}+ 2^{j-1}} 2^{3^{\alpha+1}m+2^{j+1}-3} 3^{(\alpha+1)(3^{\alpha}m+2^j)-1} m^{3^{\alpha}m+2^j}}{\bullet}\right).$$ 
	Thus, $ B_{\alpha, m, j}(z) \in  M_{k}(\Gamma_{0}(N), \chi)$ where $k$, $N$ and $\chi$ are as above. Applying Theorem \ref{thm2.5}, we obtain that the Fourier coefficients of $ B_{\alpha, m, j}(z)$ satisfies \eqref{2e1} which implies that the Fourier coefficient of $B_{\alpha, m, j}(z)$ are almost always divisible by $r= 2^j$. Hence, from $\eqref{3e6}$, we conclude that $a_{3^{\alpha}m}(n)$ are almost always divisible by $2^j$. This completes the proof of Theorem \ref{mainthm1}.
	\qed	
	
	\section{Proof of Theorem \ref{mainthm2}}
	We follow the same approach as in the proof of Theorem \ref{mainthm1}.
	Here we define
	$$C_{\alpha, m}(z):= \frac{\eta^{3}(2^3 3^{\alpha +1} m z)}{\eta(2^3 3^{\alpha +2} m z)} = \prod_{n=1}^{\infty} \frac{\left(1-q^{2^3 3^{\alpha +1} mn}\right)^3}{\left(1-q^{2^3 3^{\alpha +2} mn}\right)}$$
	and then using \eqref{3e2} for $p=3$, we get
	\begin{align*}
		C^{3^j}_{\alpha, m}(z)= \frac{\eta^{3^{j+1}}(2^3 3^{\alpha +1} m z)}{\eta^{3^j}(2^3 3^{\alpha +2} mz)} \equiv 1 \pmod {3^{j+1}}.
	\end{align*}
	Next we consider the eta-quotient
	\begin{align*}
		D_{\al, m, j}(z)
		= \frac{\eta^{3^{\alpha}m}(2^3 3^{\alpha +1} m z)}{\eta(2^33z)}	C^{3^j}_{\alpha, m}(z)
		=  \frac{\eta^{3^{\alpha}m+ 3^{j+1}}(2^3 3^{\alpha +1} m z)}{ \eta(2^33z) \eta^{3^j}(2^3 3^{\alpha +2} mz)}. 
	\end{align*}
	From \eqref{1e1}, we obtain
	\begin{align*}
		\sum_{n=0}^{\infty} a_{3^{\alpha}m}(n) q^{24n} &= \frac{(q^{2^33^{\alpha+1}m};q^{2^33^{\alpha+1}m})^{3^{\alpha}m}_{\infty}}{(q^{2^33};q^{2^33})_{\infty}}=q^{1-3^{2\al}m^2}\frac{\eta^{3^{\al}m}(2^3 3^{\alpha +1} m z)}{\eta(2^3 3 z)}.
	\end{align*}
	Therefore, from the above discussion we conclude that
	\begin{align}
		D_{\alpha, m, j}(z) \equiv \sum_{n=0}^{\infty} a_{3^{\alpha}m}(n) q^{24n+ 3^{2 \alpha}m^2 -1} \pmod{3^{j+1}}.
	\end{align}
	Hence, to prove Theorem \ref{mainthm2}, it is enough to prove that the Fourier coefficients of $D_{\alpha, m, j}(z)$ are almost always divisible by $r= 3^j$. We first prove that $D_{\alpha, m, j}(z)$ is a modular form. Using Theorem \ref{thm2.3}, we find that the level of eta quotient $D_{\alpha, m, j}(z)$ is equal to $N=2^33^{\alpha+2}m M,$ where $M$ is the smallest positive integer which satisfies 
	\begin{align}
		2^33^{\alpha+2}m M \left[ \frac{3^{\alpha}m+3^{j+1}}{2^33^{\alpha+1}m} - \frac{1}{2^3 3} - \frac{3^{j}}{2^33^{\alpha+2}m} \right]\equiv 0 \pmod{24} \implies 2^3 3^j M \equiv 0 \pmod{24}. 
	\end{align}
	Therefore $M=1$ and the level of $D_{\alpha, m, j}(z)$  is $N=2^3 3^{\alpha+2}m.$
	The cusps of $\Gamma_{0}(2^3 3^{\alpha+2}m)$ are given by fractions $\frac{c}{d}$ where $d|2^3 3^{\alpha+2}m$ and $\gcd(c,d)=1.$ By using Theorem \ref{thm2.4}, we have that $B_{\alpha, m, j} $ is holomorphic at a cusp $\frac{c}{d}$ if and only if 
	\begin{align*}
		& \left(3^{\alpha}m+3^{j+1}\right)\frac{\gcd^2(d, 2^33^{\alpha+1}m)}{2^33^{\alpha+1}m}- \frac{\gcd^2(d, 24)}{24} -3^j \frac{\gcd^2(d, 2^33^{\alpha+2}m)}{2^33^{\alpha+2}m} \geq 0 \\
		& \iff L:= \left(3^{\al+1}m+3^{j+2}\right) G_1 - 3^{\al+1}m G_2- 3^j \geq 0,
	\end{align*}
	where $G_1= \frac{\gcd^2(d, 2^3 3^{\al+1}m)}{\gcd^2(d, 2^3 3^{\al+2}m)}$ and  $G_2= \frac{\gcd^2(d, 24)}{\gcd^2(d, 2^33^{\al+2}m)}$. Let $d$ be a divisor of $2^3 3^{\al+2}m$. We write $d= 2^{r_1} 3^{r_2} t$  where $ 0 \leq r_1 \leq 3,$   $ 0 \leq r_2 \leq \al+2$ and $t|m$. We now consider the following two cases depending on $r_2$.
	
	\noindent Case 1: Let $ 0 \leq r_1 \leq 3,$ $0 \leq r_2 \leq \alpha+1$. Then $G_1= 1$ and $\frac{1}{3^{2\alpha}t^2} \leq G_2 \leq  1$. Therefore $L = 3^{\al+1}m+3^{j+2} - 3^{\alpha+1}m G_2- 3^j \geq  3^{\al+1}m+3^{j+2}-  3^{\al+1}m - 3^j = 2\cdot3^j > 0$.
	
	\noindent Case 2: Let $ 0 \leq r_1 \leq 3,$ ${r_2}=\al+2$. Then $G_1=\frac{1}{9},$ $\frac{1}{ 3^{2\alpha+2}t^2} \leq G_2 \leq  \frac{1}{9}.$  Hence, we have
	$ L= (3^{\al+1}m+3^{j+2}) G_1 - 3^{\al+1}m G_2- 3^j  \geq 3^{\alpha-1}m+ 3^j -3^{\alpha-1}m - 3 ^j=0 .$ 
	
	This proves that $D_{\alpha, m, j}(z)$ is holomorphic at every cusp $\frac{c}{d}$. Applying Theorem \ref{thm2.3}, we obtain that the weight of $D_{\al, m, j}(z)$ is $k=\frac{3^{\al}m-1}{2}+ 3^{j}$ which is a positive integer. The associated character for $D_{\alpha, m, j}(z)$ is $$\chi= \left( \frac{(-1)^{\frac{3^{\alpha}m-1}{2}+ 3^{j}} 2^{3^{\al+1}m+2\cdot3^{j+1}-3} 3^{(\al+1)3^{\al}m+(2\al+1)3^j-1} m^{3^{\al}m+2\cdot 3^j}}{\bullet}\right).$$ 
	Thus, $ D_{\alpha, m, j}(z) \in  M_{k}(\Gamma_{0}(N), \chi)$ where $k$, $N$ and $\chi$ are as above. Applying a deep result of Serre (Theorem \ref{thm2.5}), we obtain that the Fourier coefficients of $ D_{\alpha, m, j}(z)$ satisfies \eqref{2e1}. Hence the Fourier coefficients of $D_{\alpha, m, j}(z)$ are almost always divisible by $r= 3^j$. This completes the proof.
	\qed

	\section{Proof of Theorem \ref{mainthm11}}
	Let $t=p_1^{a_1}p_2^{a_2}\ldots p_m^{a_m},$ where $p_i$'s
	are primes.
	From \eqref{1e1}, we get
	\begin{small}
		\begin{equation}\label{eq701}
			\sum_{n=0}^{\infty} a_{t}(n) q^n =  \frac{(q^{t};q^{t})^{t}_{\infty}}{(q;q)_{\infty}},
		\end{equation}
	\end{small}
	Note that for any prime $p$ and positive integers $j, k$  we have
	\begin{small}
		\begin{equation}\label{eq702}
			(q;q)_{\infty}^{p^j}\equiv (q^p;q^p)_{\infty}^{p^{j-1}} \pmod{p^j} \implies 	(q^k;q^k)_{\infty}^{p^j}\equiv (q^{kp};q^{kp})_{\infty}^{p^{j-1}} \pmod{p^j}.
		\end{equation}
	\end{small}
	For a positive integer $i,$ we define $$A_i(z):= \frac{\eta(24z)^{p_i^{a_i}}}{\eta \left(24p_i^{a_i}z\right)} .$$
	Using \eqref{eq702}, we get 
	\begin{small}
		\begin{equation*}
			A_i^{p_i^j}(z):= \frac{\eta(24z)^{p_i^{a_i+j}}}{\eta \left(24p_i^{a_i}z\right)^{p_i^j}} \equiv 1 \pmod{p_i^{j+1}}.
		\end{equation*}
	\end{small}
	Define 
	$$B_{i,j,t}(z)= \left(\frac{\eta(24t z)^{t}}{\eta(24z)}\right)A_i^{p_i^j}(z) =\frac{\eta^{t}(24t z) \eta^{(p_i^{a_i+j}-1)}(24z)}{\eta^{p_i^j} \left(24p_i^{a_i}z\right)} .$$
	On modulo $p_i^{j+1},$ we get
	\begin{small}
		\begin{equation}\label{eq703}
			B_{i,j,t}(z)= \frac{\eta(24t z)^{t}}{\eta(24z)}= q^{t^2-1} \frac{(q^{24t};q^{24t})^{t}_{\infty}}{(q^{24};q^{24})_{\infty}}
		\end{equation}
	\end{small}
	Combining \eqref{eq701} and \eqref{eq703} together, we obtain
	\begin{small}
		\begin{equation}\label{eq704}
			B_{i,j,t}(z)\equiv  q^{t^2-1} \frac{(q^{24t};q^{24t})^{t}_{\infty}}{(q^{24};q^{24})_{\infty}} \equiv \sum_{n=0}^{\infty} a_{t}(n)q^{24n+t^2-1} \pmod{p_i^{j+1}}
		\end{equation}
	\end{small}
	Next, we prove that $B_{i,j,t}(z)$ is a modular form. Applying Theorem \ref{thm2.3}, we first estimate the level of eta quotient $B_{i,j,\ell}(z)$ . The level of $ B_{i,j,t}(z) $ is $N=24p_1^{a_1}p_2^{a_2}\ldots p_m^{a_m} M,$ where $M$ is the smallest positive integer which satisfies 
	\begin{small}
		\begin{align*}
			24t M \left[ \frac{t}{24t} + \frac{p_i^{a_i+j}-1}{2^3 3} - \frac{p_i^{j}}{24p_i^{a_i}} \right]\equiv 0 \pmod{24} 
			\implies t   p_i^j M\left[p_i^{a_i}- \frac{1}{p_i^{a_i}} \right]\equiv 0 \pmod{24}.
		\end{align*}
	\end{small}
	Therefore $M=24$ and the level of $B_{i,j,t}(z)$  is $N=2^6 3^2t$. The cusps of $\Gamma_{0}(2^6 3^2t)$ are given by fractions $\frac{c}{d}$ where $d|2^6 3^2t$ and $\gcd(c,d)=1.$ By using Theorem \ref{thm2.4}, we have that $B_{i,j,t}(z) $ is holomorphic at a cusp $\frac{c}{d}$ if and only if 
	\begin{small}
		\begin{align*}
			&t \frac{  \gcd^2(d, 24 t)}{24 t}+ \left(p_i^{a_i+j}-1\right) \frac{\gcd^2(d, 24)}{24} -p_i^j \frac{\gcd^2(d, 24p_i^{a_i})}{24p_i^{a_i}} \geq 0 \\
			& \iff L:= 1 + \left(p_i^{a_i+j}-1\right)  G_1 -  \frac{p_i^j}{p_i^{a_i}} G_2 \geq 0,
		\end{align*}
	\end{small}
	where $G_1= \frac{\gcd^2(d, 24)}{\gcd^2(d, 24 t )}$ and  $G_2= \frac{\gcd^2(d, 24 p_i^{a_i})}{\gcd^2(d, 24 t)}.$ 
	Let $d$ be a divisor of $ 2^6 3^2t$. We can write $d= 2^{r_1} 3^{r_2} p_i^{s} \ell$ where $ 0 \leq r_1 \leq 6$,   $ 0 \leq r_2 \leq 2$,  $ 0 \leq s \leq a_i$  and $\ell|t$ but $p_i \nmid \ell$. It is immediate that $ G_1 = \frac{1}{p_i^{2s}\ell^2}$ and $  G_2 =\frac{1}{\ell^2}$. Therefore we have
	\begin{small}
		\begin{align*}
			L&= 1+ \frac{p_i^{a_i+j}-1}{p_i^{2s}\ell^2}  -  \frac{p_i^j}{p_i^{a_i}\ell^2}= 1 + \frac{p^j(p_i^{a_i} - p_i^{2s-a_i})-1}{p_i^{2s }\ell^2}.
		\end{align*}
	\end{small}
	Note that $s\leq a_i$ implies $p^{2s-a_i}\leq p^{a_i}$. Thus $p^j(p_i^{a_i} - p_i^{2s-a_i})\geq 0$. Hence 
	\begin{small}
		\begin{align*}
			L\geq 1-\frac{1}{p_i^{2s }\ell^2}\geq 0.
		\end{align*} 
	\end{small}
	Therefore, $B_{i,j,t}(z) $ is holomorphic at every cusp $\frac{c}{d}.$
	Using Theorem \ref{thm2.3}, we compute the weight of $B_{i,j,t}(z) $ is $k= \frac{t+ \left(p_i^{a_i+j}-1\right)-p_i^j}{2}=\frac{t+ p_i^j\left(p_i^{a_i}-1\right)-1}{2}$ which is a positive integer. The associated character for $B_{i,j,t}(z) $ is $$\chi= \left( \frac{(-1)^{k} \left(24t\right)^{t} 24^{\left(p_i^{a_i+j}-1\right)} \left(24p_i^{a_i}\right)^{-p_i^j}}{\bullet}\right).$$ 
	Thus $ B_{i,j,t}(z) \in  M_{k}(\Gamma_{0}(N), \chi)$ where $k$, $N$ and $\chi$ are as above. Applying Theorem \ref{thm2.3}, we obtain that the Fourier coefficients of $B_{i,j,t}(z) $ satisfies \eqref{2e1} which implies that the Fourier coefficient of $B_{i,j,t}(z) $ are almost always divisible by $ p_i^j$. Hence, from $\eqref{eq704}$, we conclude that $a_{t}(n)$ are almost always divisible by $p_i^j$. This completes the proof of Theorem \ref{mainthm11}.

		\section{Proof of Theorems \ref{mainthmRaduSellers}, \ref{congruences}, \ref{congruences5} and \ref{congruencesrandom}}
	\subsection{An algorithmic approach by Radu and Sellers}
	We begin with recalling an algorithm developed by Radu and Sellers \cite{Radu2011} that will be used to prove Theorem \ref{mainthmRaduSellers}. Let $M$ be a positive integer and let $R(M)$ denote the set of integers sequences $r=(r_\delta)_{\delta|M}$ indexed by the positive divisors of $M$. For $r \in R(M)$ and the positive divisors $1=\delta_1<\delta_2<\cdots<\delta_{i_M}=M$ of $M$,  we set $r=(r_{\delta_1},r_{\delta_2},\dots , r_{\delta_{i_M}})$. We define $c_r(n)$  by
	\begin{align*}
		\sum\limits_{n=0}^{\infty} c_r(n)q^n:=\prod\limits_{\delta|M}(q^{\delta};q^{\delta})_{\infty}^{r_{\delta}}=\prod\limits_{\delta|M}\prod\limits_{n=1}^{\infty}(1-q^{n\delta})^{r_{\delta}}.
	\end{align*}
	Radu and Sellers \cite{Radu2011} approach to prove congruences for $c_r(n)$ modulo a positive integer reduced the number of cases that we need to check as compared with the classical method which uses Sturm's bound alone.
	
	Let $m\geq0$ and $s$ be integers. We denote by $[s]_m$ the residue class of $s$ in $\mathbb{Z}_m$ and we denote by $\mathbb{S}_m$ the set of squares in $\mathbb{Z}_m^{*}$. For $t \in \{0,1, \dots, m-1\}$ and $r\in R(M)$, the subset $P_{m,r}(t)\subseteq \{0,1, \dots, m-1\}$ is defined as
	\begin{align*}
		P_{m,r}(t):=\left\{t^{'}: \exists [s]_{24m} \textrm{ such that } t^{'}\equiv ts+\frac{s-1}{24} \sum\limits_{\delta|M} \delta r_{\delta} \pmod m\right\}.
	\end{align*}
	
	\begin{definition}
		For positive integers $m$, $M$ and $N$,  let $r=(r_\delta)\in R(M)$ and $t \in \{0,1, \dots, m-1\}$. Let $\kappa=\kappa(m):=\gcd(m^2-1,24)$ and write
		\begin{align*}
			\prod\limits_{\delta|M} \delta^{|r_{\delta}|}=2^s\cdot j,
		\end{align*}
		where $s$ and $j$ are non-negative integers with $j$ odd. The set $\Delta^{*}$ is the collection of all tuples $(m, M, N, (r_{\delta}), t)$ satisfying the following conditions.
		\begin{itemize}
			\item[(a)] Every prime divisor of $m$ is also a divisor of $N$.
			\item[(b)] If $\delta|M,$ then $\delta|mN$ for every $\delta\geqslant 1$ such that $r_{\delta} \ne 0$.
			\item[(c)] $\kappa N \sum\limits_{\delta|M} r_{\delta}mN/\delta \equiv 0 \pmod {24}$.
			\item[(d)] $\kappa N \sum\limits_{\delta|M} r_{\delta} \equiv 0 \pmod {8}$.
			\item[(e)] $\frac{24m}{\gcd(-24\kappa t-\kappa\sum\limits_{\delta|M} \delta r_{\delta}, 24m)}$divides $N$.
			\item[(f)] If $2|m$, then either ($4|\kappa N$ and $8|sN$) or ($2|s$ and $8|(1-j)N$).
		\end{itemize}
	\end{definition}
	
	For positive integers $m$, $M$ and $N$, $\gamma= \begin{bmatrix}
		a &b\\ c&d	\end{bmatrix} \in \Gamma$, $r\in R(M)$ and $a\in R(N)$, we define
	\begin{align*}
		p_{m,r}(\gamma):= \min \limits_{\lambda\in\{0,1,\dots,m-1\}} \frac{1}{24} \sum \limits_{\delta|M} r_{\delta}\frac{\gcd^2(\delta a + \delta \kappa \lambda c,mc)}{\delta m}
	\end{align*}
	and
	\begin{align*}
		p^{*}_{a}(\gamma):=\frac{1}{24} \sum \limits_{\delta|N} a_{\delta}\frac{\gcd^2(\delta,c)}{\delta}.
	\end{align*}
	
	The following lemma is given by Radu \cite[Lemma $4.5$]{radu2009}.
	\begin{lemma}\label{finite check}
		Let $u$ be a positive integer, $(m, M, N, (r_{\delta}), t)\in \Delta^{*}$ and $a=(a_{\delta})\in R(N)$. Let $\{\gamma_1, \gamma_2, \dots, \gamma_n\} \subseteq \Gamma$ denote a complete set of representatives of the double cosets of $\Gamma_0(N)\backslash \Gamma / \Gamma_{\infty}$. Assume that $p_{m,r}(\gamma_i) +p^{*}_{a}(\gamma_i)\geq 0$ for all $1 \leq i \leq n$. Let $t_{\min}= \min_{t^{'} \in P_{m,r}(t)} t^{'}$ and
		\begin{align*}
			\nu:= \frac{1}{24} \left\{\left(\sum \limits_{\delta|M} r_{\delta}+\sum \limits_{\delta|N} a_{\delta}\right) [\Gamma: \Gamma_0(N)]-\sum \limits_{\delta|N} \delta a_{\delta} \right\} - \frac{1}{24m}\sum \limits_{\delta|M} \delta r_{\delta} -\frac{t_{\min}}{m}.
		\end{align*}
		If the congruence $c_r(mn+t^{'}) \equiv 0 \pmod u$ holds for all $t^{'} \in P_{m,r}(t)$ and $ 0\leq n \leq \lfloor \nu \rfloor$, then $c_r(mn+t^{'}) \equiv 0 \pmod u$ holds for all $t^{'} \in P_{m,r}(t)$ and $n \geq 0$.
	\end{lemma}
	
	The next lemma is given by Wang \cite[Lemma $4.3$]{wang2017}. This result gives the complete set of representatives of the double cosets in  $\Gamma_0(N)\backslash \Gamma / \Gamma_{\infty}$ when $N$ or $\frac{N}{2}$ is a square-free integer.
	
	\begin{lemma}\label{square-free}
		If $N$ or $\frac{N}{2}$ is a square-free integer, then
		\begin{align*}
			\bigcup_{\delta|N} \Gamma_0(N) \begin{bmatrix}
				1 &0\\ \delta &1	\end{bmatrix} \Gamma_{\infty} =\Gamma.
		\end{align*}
	\end{lemma}
	
	\subsection{Proof of Theorem \ref{mainthmRaduSellers}}
	For an integer $m$, a prime $p\geqslant 3$ and $t\in \{0,1,\dots,m-1\}$, we recall that
	\begin{align*}
		\kappa:=&\kappa(m)=\gcd(m^2-1,24),\\
		\hat {p}:=&\frac{p^2-1}{24},\\
		A_t:=&A_t(m,p)=\frac{24m}{\gcd(-\kappa(24 t+p^2-1),24m)} =\begin{cases}
			\frac{m}{\gcd(-\kappa (t+\hat{p}),m)} & \textrm{ if } p>3,\\
			\frac{3m}{\gcd(-\kappa (3t+1),3m)}& \textrm{ if } p=3.,
		\end{cases}\\
		\epsilon_2:=&\epsilon_2(m,p)=\begin{cases}
			\frac{1-(-1)^{\frac{p-1}{2}}}{2} & \textrm{ if } 2| m,\\
			0 & \textrm{ if } 2\nmid m,
		\end{cases}\\
		\epsilon_p:=&\epsilon_p(m)=\begin{cases}
			1 & \textrm{ if } p\nmid m,\\
			0 & \textrm{ if } p|m.
		\end{cases}
	\end{align*}

	We now prove the following three results specific to the proof of Theorem \ref{mainthmRaduSellers}.
	
	\begin{lemma} \label{lemdelta}
		Let $p\geq 3$ be a prime number. For a positive integer $g$, let $e_1, e_2, \ldots, e_g$ be non-negative integers and let $p_1,p_2, \ldots, p_{g}$ be prime numbers. Let  $$(m,M,N,r,t)=(p_1^{e_1} p_2^{e_2} \cdots p_g^{e_g},p,2^{\epsilon_2} p^{\epsilon_p} p_1 p_2 \cdots p_{g},r=(r_1=-1,r_p=p),t)$$ where $t \in \{0,1,\ldots, m-1\}$ is such that $A_t|N$. Then $(m,M,N,r,t) \in \Delta^{*}$.
	\end{lemma}
	\begin{proof}
		We first note that
		\begin{align*}
			\gcd(m^2-1,8)=\begin{cases}
				8 & \textrm{ if } $m$ \textrm{ is odd,}\\ 
				1 & \textrm{ if } $m$ \textrm{ is even}\\ 
			\end{cases}
			\quad\textrm{and}\quad
			\gcd(m^2-1,3)=\begin{cases}
				3 & \textrm{ if } 3 \nmid m,\\ 
				1 & \textrm{ if } 3|m.\\ 
			\end{cases}
		\end{align*}
		Therefore
		\begin{align*}
			\kappa=\gcd(m^2-1,24)=\begin{cases}
				24& \textrm{ if } \gcd(m,6)=1,\\ 
				8 & \textrm{ if } \gcd(m,6)=3,\\ 
				3 & \textrm{ if } \gcd(m,6)=2,\\ 
				1 & \textrm{ if } \gcd(m,6)=6.\\ 
			\end{cases}
		\end{align*}
		It is immediate that the conditions (a) and (b) in the definition of $\Delta^{*}$ are satisfied. Since $M=p$ and $r=(r_1=-1,r_p=p)$, we see that $\sum\limits_{\delta|M} r_{\delta}mN/\delta=r_1mN+r_{p}mN/p=0$ and therefore (c) is also satisfied. Next, we note that $\kappa N\sum\limits_{\delta|M} r_{\delta} =\kappa N (p-1)$ and $N(p-1)\equiv0\pmod 4$. We have $\kappa N(p-1)\equiv 0 \pmod 8$ if $\gcd(m,6)=1$ or $3$. Further, if $\gcd(m,6)=2$ or $6$, then $2|m$ gives $2^{1+\epsilon_2}|N$ and so it follows that $N(p-1)\equiv 0 \pmod 8$. Thus (d) holds. From the conditions on $t$, we see that (e) is also satisfied. For (f), we see that $\prod\limits_{\delta|M} \delta^{|r_{\delta}|}=p^p$ gives $s=0$ and $j=p^p$. As observed earlier in the justification for (c), we have $N(p-1)\equiv 0 \pmod 8$ if $2|m$. Thus if $2|m$, then $(1-p)|(1-j)$ implies $8|(1-j)N$. This completes the proof. 
	\end{proof}
	
	\begin{lemma} \label{pmrpositive}
		Let $p\geq 3$ be a prime number. For a positive integer $m$, $M=p$, $r=(r_1=-1,r_p=p)$ and $\gamma= \begin{bmatrix}
			a &b\\ c&d	\end{bmatrix} \in \Gamma$, we have
		\begin{align*}
			p_{m,r}(\gamma)= \min \limits_{\lambda\in\{0,1,\dots,m-1\}} \frac{1}{24} \sum \limits_{\delta|M} r_{\delta}\frac{\gcd^2(\delta a + \delta \kappa \lambda c,mc)}{\delta m} \geqslant 0.
		\end{align*}
	\end{lemma}
	\begin{proof} We note that
		\begin{align*}
			\sum \limits_{\delta|M} r_{\delta}\frac{\gcd^2(\delta a + \delta \kappa \lambda c,mc)}{\delta m}=p\frac{\gcd^2(p a + p \kappa \lambda c,mc)}{p m}
			-\frac{\gcd^2( a + \kappa \lambda c,mc)}{ m}
		\end{align*}
		Since $\gamma \in \Gamma=SL_2(\Z)$, we have $ad-bc=1$ which implies that $\gcd(a,c)=1$. Thus it follows that $\gcd( a +  \kappa \lambda c,c)=1$. Therefore it is enough to prove that
		\begin{align*}
			G:={\gcd}^2(p a + p \kappa \lambda c,mc)
			-{\gcd}^2( a + \kappa \lambda c,m)\geq 0
		\end{align*}
		for each $\lambda\in\{0,1,\dots,m-1\}$. 
		
		Let $\lambda\in\{0,1,\dots,m-1\}$ be fixed. We consider the two cases $p|c$ and $p\nmid c$ separately.
		
		\noindent\textbf{Case 1: $p|c$.} We set $c_p=\frac{c}{p}$. We observe that $\gcd(a,c)=1$ implies that $\gcd(a,c_p)=1$ which in turn implies $\gcd( a +  \kappa \lambda c,c_p)=\gcd( a +  \kappa \lambda p c_p,c_p)=1$. Hence we have 
		\begin{align*}
			G&=p^2{\gcd}^2( a +  \kappa \lambda c,mc_p)
			-{\gcd}^2( a + \kappa \lambda c,m)\\
			&=p^2{\gcd}^2( a +  \kappa \lambda c,m)
			-{\gcd}^2( a + \kappa \lambda c,m)>0.
		\end{align*}
		\textbf{Case 2: $p\nmid c$.} In this case, we have
		\begin{align*}
			G&={\gcd}^2( p(a +  \kappa \lambda c),m)
			-{\gcd}^2( a + \kappa \lambda c,m).
		\end{align*}
		If $p\nmid m$, then $\gcd( p(a +  \kappa \lambda c),m)=\gcd( a +  \kappa \lambda c,m)$ and thus $G=0$. We now assume that $p|m$. We set $m_p=\frac{m}{p}\in \Z$. Then
		\begin{align*}
			G&=p^2{\gcd}^2( a +  \kappa \lambda c,m_p)
			-{\gcd}^2( a + \kappa \lambda c,m).
		\end{align*}
		Let $d=\gcd( a + \kappa \lambda c,m)$. Let ord$_p(n)$ denote the highest exponent of $p$ dividing a positive integer $n$. It is clear that ord$_p(d)\leq$ord$_p(m)=$ord$_p(m_p)+1$. If ord$_p(d)\leq$ord$_p(m_p)$, then $G=d^2(p^2-1)>0$. If ord$_p(d)=$ord$_p(m_p)+1$, then $\gcd( a + \kappa \lambda c,m_p)=\frac{d}{p}$ and therefore $G=p^2(\frac{d}{p})^2-d^2=0$. This completes the proof.
	\end{proof}
	
	\begin{lemma}\label{lembound}
		Let $p\geq 3$ be a prime and let $u$ be an integer. Let $(m,M,N,r,t)$ be as defined in Lemma \ref{lemdelta}. Let $t_{\min}= \min_{t^{'} \in P_{m,r}(t)} t^{'}$.	If the congruence $a_p(mn+t^{'}) \equiv 0 \pmod u$ holds for all $t^{'} \in P_{m,r}(t)$ and $ 0\leq n \leq \left\lfloor \frac{2^{\epsilon_2}(p+1)^{\epsilon_p}(p-1)(p_1+1)(p_2+1)\cdots (p_{g}+1)}{24} \right\rfloor$, then $a_p(mn+t^{'}) \equiv 0 \pmod u$ holds for all $t^{'} \in P_{m,r}(t)$ and $n \geq 0$.
	\end{lemma}
	\begin{proof}
		It is enough to show that the assumptions of Lemma \ref{finite check} are satisfied and that the upper bound in Lemma \ref{finite check} is less than or equal to $\left\lfloor \frac{2^{\epsilon_2}(p+1)^{\epsilon_p}(p-1)(p_1+1)(p_2+1)\cdots (p_{g}+1)}{24} \right\rfloor$. For $\delta|N$, we set $\gamma_{\delta}=\begin{bmatrix} 1 &0\\ \delta&1 \end{bmatrix}$. Since $\epsilon_2= 0$ or $1$, $N$ or $\frac{N}{2 }$ is a square-free integer. Thus Lemma \ref{square-free} implies that $\{\gamma_{\delta}: \delta|N\}$ forms a complete set of double coset representatives of $\Gamma_0(N)\backslash \Gamma / \Gamma_{\infty}$. Lemma \ref{pmrpositive} implies that $p_{m,r}(\gamma_{\delta})\geq 0$ for each $\delta|N$.	Therefore we take $a_{\delta}=0$ for each $\delta|N$, that is, $a=(0,0, \ldots, 0)\in R(N)$ and hence $p_{m,r}(\gamma_{\delta}) +p_{a}^{*}(\gamma_{\delta})\geq0$ for each $\delta|N$. Since $t_{\min}\geq 0$, we have
		\begin{align*}
			\lfloor \nu \rfloor&= \left \lfloor \frac{(p-1)}{24}[\Gamma:\Gamma_0(N)] - \frac{p^2-1}{24m} -\frac{t_{\min}}{m} \right \rfloor\\
			& \leqslant \left \lfloor \frac{(p-1)}{24}[\Gamma:\Gamma_0(N)] - \frac{p^2-1}{24m} \right \rfloor\\
			&= \left \lfloor \frac{2^{\epsilon_2}(p+1)^{\epsilon_p}(p-1)(p_1+1)(p_2+1)\cdots (p_{g}+1)}{24} - \frac{p^2-1}{24m} \right \rfloor.
		\end{align*}
	\end{proof}
	
	\begin{proof}[Proof of Theorem \ref{mainthmRaduSellers}]
		We get from \eqref{1e1} that
		\begin{align*}
			\sum\limits_{n=0}^{\infty} a_p(n) q^n= \frac{(q^p;q^p)^{p}_{\infty}}{(q;q)_{\infty}}.
		\end{align*}
		Let $(m,M,N,r,t)=(p_1^{e_1} p_2^{e_2} \cdots p_g^{e_g},p,2^{\epsilon_2} p^{\epsilon_p} p_1 p_2 \cdots p_{g},r=(r_1=-1,r_p=p),t)$ be such that $t \in \{0,1,\ldots, m-1\}$ and $A_t|N$. Then Lemma \ref{lemdelta} implies $(m,M,N,r,t) \in \Delta^{*}$. Since $\sum_{\delta|M}\delta r_{\delta}=p^2-1$, we see that $P_{m,r}(t)=P(t)$. Now the assertion follows from Lemma \ref{lembound}.
	\end{proof}

	\subsection{Proof of Theorem \ref{congruences}}
	We begin with the proof of the first congruence. We take $p=5$, $m=3\cdot 5=15$ and $t=6$. Using Sage \cite{Sage}, we find that $P(6)=\{6,10,12,13\}$. We compute that the upper bound in Corollary \ref{coropq} is less than $4\hat{p}=4$. Using Mathematica, we verify that $a_5(15n +t) \equiv 0 \pmod {3}$ for $t\in\{6,10,12,13\}$ and $0\leq n < 4$. Now the first congruence follows immediately from Corollary \ref{coropq}.
	
	A similar approach can be used to prove the other congruences. In particular, we set $B_p=4\hat{p}$ and we put the values of $p$, $m$, $t$, $P(t)$ and $B_p$ in the following table.
	\begin{center}
		\begin{tabular}{|l|l|l|l|l|}
			\hline
			$p$ & $m$ & $t$ & $P(t)$ & $B_p$\\
			\hline
			7 & 21 & 3,8 & $P(3)=\{3,15,18\}$, $P(8)=\{8,11,17\}$ & 8 \\
			\hline
			\multirow{2}{*}{11} & \multirow{2}{*}{33} & \multirow{2}{*}{3,11} & $P(3)=\{3,12,24,27,30\}$, & \multirow{2}{*}{20}\\
			& & & $P(11)=\{11,20,26,29,32\}$& \\
			\hline
			\multirow{2}{*}{13} & \multirow{2}{*}{39} & \multirow{2}{*}{3,7} & $P(3)=\{3,9,15,18,33,36\}$,& \multirow{2}{*}{28} \\
			& & &  $P(7)=\{7,10,16,22,28,31\}$ & \\
			\hline
			\multirow{2}{*}{17} & \multirow{2}{*}{51} & \multirow{2}{*}{10,14} &$P(10)=\{10,16,19,25,28,34,46,49\}$& \multirow{2}{*}{48}\\
			&&& $P(14)=\{14,20,23,26,35,38,41,47\}$& \\
			\hline
			\multirow{2}{*}{19} & \multirow{2}{*}{57} & \multirow{2}{*}{7,14} & $P(7)=\{7,16,19,22,25,31,37,52,55\}$, & \multirow{2}{*}{60} \\
			&&& $P(14)=\{14,17,26,35,38,41,44,50,56\}$ &\\
			\hline
			\multirow{2}{*}{23} & \multirow{2}{*}{69} & \multirow{2}{*}{3,16} &  $P(3)=\{3,9,27,30,33,36,42,48,51,60,63\}$, & \multirow{2}{*}{88} \\
			&&& $P(16)=\{16,22,31,34,43,46,52,58,61,64,67\}$ &\\
			\hline
		\end{tabular}
	\end{center}
	For each $p$ and $m$ given in the above table, we verify, using Mathematica, that $a_p(mn +t^{'}) \equiv 0 \pmod {3}$ for $t^{'}\in P(t)$ and $0\leq n < B_p$ for respective values of $t$ and $B_p$. Now the congruences follows immediately from Corollary \ref{coropq}. This completes the proof.

	\subsection{Proof of Theorem \ref{congruences5}}
	We start by proving the first congruence. We take $p=7$, $m=5\cdot 7=35$ and $t=4$. Using Sage \cite{Sage}, we find that $P(4)=\{4,17,22,24,29,32\}$. We compute that the upper bound in Corollary \ref{coropq} is less than $6\hat{p}=12$. Using Mathematica, we verify that $a_7(35n +t) \equiv 0 \pmod {5}$ for $t\in\{4,17,22,24,29,32\}$ and $0\leq n < 12$. Now the first congruence follows immediately from Corollary \ref{coropq}.
	
	A similar approach can be used to prove the other congruences. In particular, we set $B_p=6\hat{p}$ and we put the values of $p$, $m$, $t$, $P(t)$ and $B_p$ in the following table.
	\begin{center}
		\begin{tabular}{|l|l|l|l|l|}
			\hline
			$p$ & $m$ & $t$ & $P(t)$ & $B_p$\\
			\hline
			\multirow{2}{*}{11} & \multirow{2}{*}{55} & \multirow{2}{*}{4,7} & $P(4)=\{4, 9, 11, 21, 26, 29, 31, 44, 51, 54\}$, & \multirow{2}{*}{30}\\
			& & & $P(7)=\{7, 18, 22, 32, 33, 37, 42, 43, 48, 53\}$& \\
			\hline
			17 & 85 & 4 &$P(4)=\{4, 7, 9, 14, 24, 37, 47, 52, 54, 57, 64, 69, 72, 74, 77, 82\}$& 72\\
			\hline
		\end{tabular}
	\end{center}
	For each $p$ and $m$ given in the above table, we verify, using Mathematica, that $a_p(mn +t^{'}) \equiv 0 \pmod {5}$ for $t^{'}\in P(t)$ and $0\leq n < B_p$ for respective values of $t$ and $B_p$. Now the congruences follows immediately from Corollary \ref{coropq}. This completes the proof.

	\subsection{Proof of Theorem \ref{congruencesrandom}}
	We first prove \eqref{18e2}. We take $p=5$, $m=7^2$ and $t\in\{6,20\}$. Using Sage \cite{Sage}, we find that $P(6)=\{6,13,27\}$ and $P(20)=\{20,34,41\}$. We compute that the upper bound in Theorem \ref{mainthmRaduSellers} is less than $\lfloor \frac{6\cdot 4\cdot 8}{24} \rfloor=8$. Using Mathematica, we verify that $a_5(49n +t) \equiv 0 \pmod {3}$ for $t\in\{6,13,20,27,34,41\}$ and $0\leq n < 8$. Now \eqref{18e2} follows immediately from Theorem \ref{mainthmRaduSellers}.
	
	A similar method can be used to prove \eqref{18e1} and \eqref{18e3}-\eqref{18e7}. In particular, we set $B_{p,m}=\left\lfloor \frac{2^{\epsilon_2}(p+1)^{\epsilon_p}(p-1)(p_1+1)(p_2+1)\cdots (p_{g}+1)}{24} -\frac{(p^2-1)}{24m} \right\rfloor$ and we put the values of $p$, $m$, $t$, $P(t)$ and $B_{p,m}$ in the following table.
	\begin{center}
		\begin{tabular}{|l|l|l|l|l|}
			\hline
			$p$ & $m$ & $t$ & $P(t)$ & $B_{p,m}$\\
			\hline
			5 & 8 & 3 & $P(3)=\{3\}$& 2 \\
			7 & 25 & 3,8 & $P(3)=\{3,18\}$, $P(8)=\{8,13\}$ & 11 \\
			11 & 44 & 9 & $P(9)=\{9,21,29,33,37\}$ & 29 \\
			13 & 12 & 3 & $P(3)=\{3\}$ & 5 \\
			13 & 64 & 25 & $P(25)=\{25\}$ & 20 \\
			19 & 76 & 3 & $P(3)=\{3, 7, 19, 31, 35, 55, 63, 71, 75\}$ & 89 \\
			\hline
		\end{tabular}
	\end{center}
	For each $p$ and $m$ given in the above table, we verify, using Mathematica, that $a_p(mn +t^{'}) \equiv 0 \pmod {3}$ for $t^{'}\in P(t)$ and $0\leq n \leq B_{p,m}$ for respective values of $t$ and $B_{p,m}$. Now the congruences follows immediately from Corollary \ref{coropq}. This completes the proof.
	

	\section{Conclusion}
	By Radu and Seller's method, for a prime $p\geq3$, we obtain an algorithm in Theorem \ref{mainthmRaduSellers} for congruences of the type $a_p(mn+t) \equiv 0 \pmod u$ when $m=p_1^{e_1} p_2^{e_2}\cdots p_g^{e_g}$ where $p_i$ are prime numbers and $e_1, e_2, \cdots, e_g$ are non-negative integers. Since we have obtained the density results for divisibility of $a_{3^{\alpha}m}(n)$ by powers of $2$ and $3$, and divisibility of $a_{t}(n)$, where $t=p_1^{a_1}p_2^{a_2}\ldots p_m^{a_m}$ with $p_i \geq 5$, by prime powers of $p_i^j$ in Theorems \ref{mainthm1}, \ref{mainthm2} and \ref{mainthm11}, it will be interesting to see more algebraic results in the same direction.


\end{document}